\documentclass[12pt]{amsart}
\usepackage{amsfonts}
\usepackage{amssymb}
\usepackage{color}
\usepackage{graphicx}
\usepackage{url}
\usepackage[colorlinks=true, allcolors=blue]{hyperref}
\usepackage{todonotes}

\DeclareSymbolFont{bbold}{U}{bbold}{m}{n}
\DeclareSymbolFontAlphabet{\mathbbold}{bbold}

\DeclareMathOperator{\mesh}{mesh}
\DeclareMathOperator{\diam}{diam}
\DeclareMathOperator{\confdim}{Cdim}
\DeclareMathOperator{\hdim}{Hdim}
\DeclareMathOperator{\ARConfdim}{ARConfdim}
\DeclareMathOperator{\interior}{int}

\newcommand{\Julia}{\mathcal{J}}
\newcommand{\hCbb}{\hat{\mathbb{C}}}

\newcommand{\vep}{\varepsilon}

\newcommand{\bbN}{\mathbb{N}}
\newcommand{\N}{\mathbb{N}}
\newcommand{\R}{\mathbb{R}}

\newcommand{\Y}{\mathbb{Y}}

\newcommand {\UUU}{\mathcal{U}}

\newcommand{\wtU}{\widetilde{U}}

\newcommand{\roundness}{\mathrm{round}}

\newcommand{\angela}[1]{}

\newtheorem{thm}{Theorem}[section]
\newtheorem{cor}[thm]{Corollary}
\newtheorem{lem}[thm]{Lemma}
\newtheorem{prop}[thm]{Proposition}

\newtheorem{theorem}{Theorem}
\newtheorem{corollary}[theorem]{Corollary}

\theoremstyle{definition}
\newtheorem{defn}{Definition}[section]

\theoremstyle{remark}

\begin{document}
    \author{Insung Park}
   \address{The Institute for Computational and Experimental Research in Mathematics, Providence, RI, 02903 USA}
    \email{insung\_park@brown.edu}
    
    \author{Angela Wu}
   \address{Department of Mathematics, Indiana University, Bloomington, IN 47405 USA}
    \email{angelawu0312@gmail.com}
   \title[Quasi-self-similar fractals containing ``Y'' ]{Quasi-self-similar fractals containing ``Y'' have dimension larger than one}

   % Note that the short title for running heads     goes in square
   % brackets.  This is optional.  The long title goes in curly
   % braces.  In the long title, line breaks are indicated by \\.

   \begin{abstract}
   Suppose $X$ is a compact connected metric space and $f: X \to X$ is a metric coarse expanding conformal map in the sense of Ha\"issinsky-Pilgrim. We show that if $X$ contains a homeomorphic copy of the letter ``Y'', then the Hausdorff dimension of $X$ is greater than one. As an application, we show that for a semi-hyperbolic rational map $f$ its Julia set $\Julia_f$ is quasi-symmetric equivalent to a space having Hausdorff dimension 1 if and only if $\Julia_f$ is homeomorphic to a circle or a closed interval. 
    \end{abstract}

   % AMS subject classifications (used in AMS journals)
   %\subjclass{???}

   % AMS keywords (used in AMS journals)
   %\keywords{cxc, Hausdorff dimension, conformal dimension}

   % acknowledge support, etc
   %\thanks{This research was partially supported by Simons Collaboration Grant for Mathematicians Number 615022 and by Indiana University.}

   % dedication
   %\dedicatory{Dedicated to Professor Donald Knuth on the occasion
     %of his $100$th birthday}

   % today's date, or fill in whatever date you prefer
   %\date{\today}

% This ends the top matter information.
   \maketitle
    %\tableofcontents 
\section{Introduction}

In the study of dynamical systems for rational maps on the Riemann sphere, there are fractals called Julia sets on which the dynamics is more chaotic than on the complements, called Fatou sets. If a Julia set is connected, then its Hausdorff dimension is often greater than one, see \cite{Zdunik_DimMaxMsr}, \cite{Przytycki_DimBdry}, \cite{PZ_DimPoly}. In this vein, we also discuss a family of Julia sets whose Hausdorff dimensions are greater than one. In fact, we obtain a stronger assertion in the following sense: Any quasi-symmetric deformations of those Julia sets still have Hausdorff dimension greater than one.
\medskip

Let $X$ denote a compact connected metric space and $f: X \to X$ a continuous map, so that the iterations of $f$ define a topological dynamical system on $X$. For $e_1=(1,0,0),e_2=(0,1,0),$ and $e_3=(0,0,1)$, let
\[ \Y=[0,e_1]\cup [0,e_2]\cup[0,e_3] \subset \mathbb{R}^3 \]
denote the metric space which is the union of three unit Euclidean segments meeting at the point $o=(0,0,0)$, equipped with its length metric. If $X$ contains a homeomorphic copy of $\Y$, we write $\Y \hookrightarrow X$.

The following is our main theorem. We denote by $\hdim(X)$ the Hausdorff dimension of $X$.

\begin{theorem}
\label{theorem:main}
Let $X$ be a compact connected locally connected metrizable space. Suppose $f: X \to X$ is a topological coarse expanding conformal dynamical system. If $\Y \hookrightarrow X$, then $\hdim(X,d)>1$ for any metric $d$ in the canonical gauge $\mathcal{G}(f)$.
\end{theorem}

From Theorem \ref{theorem:main} we have the following corollary.

\begin{corollary}
\label{cor:main}
Let $X$ and $f$ be as in Theorem \ref{theorem:main}. If $\Y \hookrightarrow X$ and $\confdim((X,\mathcal{G}(f)))=1$ then the conformal dimension is not attained by any metric in $\mathcal{G}(f)$.
\end{corollary}

Let us define the terminologies sued in the statement of Theorem \ref{theorem:main}.

\subsection*{Quasi-symmetries and conformal dimensions}
Let $(X,d_X)$ and $(Z,d_Z)$ be compact metric spaces. A homeomorphism $h: X \to Z$ is a {\em quasi-symmetry} if there is a homeomorphism $\eta: [0,\infty) \to [0,\infty)$ such that for every $x,y,z \in X$ with $x \neq z$, we have 
\[ \frac{d_Z(h(x), h(y))}{d_Z(h(x), h(z))} \leq \eta\left( \frac{d_X(x,y)}{d_X(x,z)} \right). \]
One can think of quasi-symmetries as homeomorphisms that do not distort shapes much.

We say that two metric spaces $X$ and $Z$ are {\em quasi-symmetrically equivalent}, or simply {\em quasi-symmetric}, and write $X \sim_{\mathrm{qs}} Z$, if there is a quasi-symmetric homeomorphism between $X$ and $Z$. Quasi-symmetries define an equivalence relation on the collection of metric spaces. 

We often consider different metrics on the same topological space $M$. In this case, we say that two metrics $d_1$ and $d_2$ on $M$ are {\em quasi-symmetric} if the identity map $\mathrm{id}:(M,d_1)\to(M,d_2)$ is a quasi-symmetry.

For a compact metric space $X$, the \emph{conformal dimension} of $X$, denoted by $\confdim(X)$, is defined as the infimal Hausdorff dimension in the quasi-symmetric equivalence class containing $X$, i.e.,
\[
    \confdim(X) = \inf\{ \hdim(Z) : X \sim_{\mathrm{qs}} Z\}.
\]
We say that $X$ {\em attains} its conformal dimension if the infimum in the definition of the conformal dimension of $X$ is realized as the minimum.

\subsection*{Coarse expanding conformal dynamical systems}
We abbreviate {\em coarse expanding conformal} as ``cxc''. The idea of cxc maps were introduced by P.\@ Ha\"{i}ssinsky and K.\@ Pilgrim in \cite{kmp:ph:cxci}. We give precise definitions in Section \ref{sec:Background}.

We say that a topological dynamical system $f:X\to X$ on a topological space $X$ is {\em topological cxc} if it satisfies 3 properties (Expansion), (Degree), and (Irreducibility): (Expansion) and (Irreducibility) are standard conditions for topological dynamical systems. The property (Degree) prohibits periodic and more generally recurrent periodic branch points, which characterizes semi-hyperbolic rational maps in complex dynamics. We remark that we do not use a metric to define topological cxc maps.

Now let us assume that $X$ is equipped with a metric $d_X$. We say that $f:X \to X$ is {\em metric cxc} if (1) $f$ is topological cxc, and (2) $f$ satisfies metric conditions defined in such a way that the {\em Sullivan's Principle of the Conformal Elevator} holds. It guarantees that the metric space $(X,d_X)$ is {\em quasi-self-similar} in a sense that any arbitrarily small piece of $X$ looks similar to a large piece of $X$ up to some bounded error. More precisely, a metric space $X$ is called quasi-self-similar if there is a $k \geq 1$ and $r_0 > 0$ such that for any ball $B \subset X$ with radius $r < r_0$ there exists a map $E_B:B \to X$ satisfying
\[
\frac{1}{k} \cdot \frac{r_0}{r} \cdot d_X(x,y) \leq d_X(E_B(x), E_B(y)) \leq k \cdot \frac{r_0}{r} \cdot d_X(x,y)
\]
for every $x,y \in B$ \cite[p.\@ 42]{Sullivan_note}. 

Although the definition of cxc maps takes motivation in complex dynamics, cxc maps also include self-maps of manifolds whose iterates are uniformly quasi-regular. Being topological or metric cxc is preserved by quasi-symmetric conjugacies and taking products. See \cite{HP_CxcExample} for more examples of cxc maps.

\subsection*{Canonical gauge and visual metrics of topological cxc maps}
Let $f:X\to X$ be a topological cxc map. If $f$ is metric cxc for two different metrics $d$ and $d'$ on $X$, then $(X,d)$ and $(X,d')$ are quasi-symmetric. See Proposition \ref{prop:CXCConj}. Hence, there exists a unique quasi-symmetric class $\mathcal{G}(f)$, called the {\em canonical gauge of $f$}, consisting of metrics for which $f$ is metric cxc. In this sense, the conformal dimension for $\mathcal{G}(f)$, i.e., the infimal Hausdorff dimension of metrics in $\mathcal{G}(f)$, is a natural invariant for topological cxc maps.

For a topological cxc map $f:X\to X$, we can construct metrics, called {\em visual metrics}, which are contained in $\mathcal{G}(f)$. One advantage of having visual metrics is the following: Equipping $X$ with a visual metric, we may assume that the $n$-th preimages of a ball are quasi-balls with uniformly bounded distortion whose radii decay exponentially fast about $n>0$.

\subsection*{Obstruction for hdim $\le$ 1} J.\@ Azzam proved that if a compact connected metric space $X$ is {\em antenna-like}, i.e., having ``sufficiently many'' copies of $\Y$ with uniformly bounded distortion, then $\hdim(X)>1$ \cite{Azzam_WigglySpace}. See Section \ref{sec:Antenna} for precise definitions.

Since being antenna-like is preserved by quasi-symmetries, to show Theorem \ref{theorem:main} it suffices to prove that $X$ with a visual metric is antenna-like.  Starting from one copy of $\Y$ in $X$, we first find many other copies of $\Y$ by using properties of topological cxc maps, and then we show these copies are not distorted much by using the properties of metric cxc maps. 

\subsection*{Families of rational maps}
In complex dynamics with one variable, we consider rational maps $f:\hCbb\to\hCbb$, which are holomorphic maps from the Riemann sphere $\hCbb$ to itself. In this article, we discuss three sub-families of rational maps. For a rational map $f$,
\begin{itemize}
    \item $f$ is {\em semi-hyperbolic} if $f$ does not have any parabolic periodic point nor any recurrent critical point;
    \item $f$ is {\em sub-hyperbolic} if any critical point is periodic, preperiodic, or attracted to an attracting periodic cycle; and
    \item $f$ is {\em post-critically finite} if any critical point is periodic or preperiodic.
\end{itemize}
Post-critically finiteness implies sub-hyperbolicity, and sub-hyperbolicity implies semi-hyperbolicity. Every Fatou component of a semi-hyperbolic rational map is in an attracting or a super-attracting basin \cite{Mane_ThmFatou}. Any sub-hyperbolic rational map can be obtained from a unique post-critically finite rational map by quasi-conformal deformation on the Fatou set, which deforms each super-attracting basin to an attracting basin \cite{McM_PCFCenter}.

\subsection*{Attainment of (Ahlfors-regular) conformal dimensions}

For applications of conformal dimensions to geometric group theory or complex dynamics, we use a slight variant of the conformal dimension, called the {\em Ahlfors-regular conformal dimension}. A metric space $X$ is {\em Ahlfors-regular} if round sets of $X$ well behave with the $\delta$-Hausdorff measure $\mathcal{H}^\delta$ where $\delta=\hdim(X)$, in a sense that there exists $C>1$ such that for any $0<r<\frac{1}{2}\cdot \diam(X)$ and $x\in X$ we have
\[
\frac{1}{C}\cdot {r^\delta} \le \mathcal{H}^\delta(B(x,r)) \le C \cdot {r^\delta}.
\]
For a compact metric space $X$, the Ahlfors-regular conformal dimension, denoted by $\ARConfdim(X)$, is defined by
\[
    \ARConfdim(X) = \inf\{ \hdim(Z) : X \sim_{\mathrm{qs}} Z, Z ~\text{is~Ahlfors-regular} \}
\]

M.\@ Bonk and B.\@ Kleiner showed that the attainment of conformal dimensions characterizes lattices of hyperbolic isometry groups.

\begin{thm}[Bonk-Kleiner \cite{BK:QMgroupaction}]\label{thm:BK_Rigidity}
Suppose that $G$ is a hyperbolic group and the boundary at infinity $\partial_\infty G$ has topological dimension $n$. If $\ARConfdim(\partial_\infty G)=n$ and attained, then $G$ is (up to finite index) the fundamental group of a closed hyperbolic $n$-manifold.
\end{thm}

A similar result for self-maps of the 2-sphere was established by P.\@ Ha\"{i}ssinsky and K.\@ Pilgrim.

\begin{thm}[Ha\"{i}ssinsky-Pilgrim \cite{HP_MinConfDimCXC}] Suppose $f:X\to X$ is a metric cxc map for a metric space $X$ homeomorphic to the 2-sphere. If $\ARConfdim(X)=\hdim(X)$, then $f$ is topologically conjugate to (1) a semi-hyperbolic rational map if $\ARConfdim(X)=2$ or (2) a Latt\'{e}'s map (with additional properties) if $\ARConfdim(X)>2$.
\end{thm}

A semi-hyperbolic rational map $f$ acts on its Julia set $\Julia_f$ as a metric cxc system. If the semi-hyperbolic rational map is a polynomial and the Julia set $\Julia_f$ is connected, then the conformal dimension of $\Julia_f$ equals 1, see \cite{Carrasco_ConfdimCombiMod}, \cite{Kinneberg:John}. 

% \begin{corollary}
% \label{corollary:poly} 
% Let $f$ be a semi-hyperbolic polynomial. If the Julia set $\Julia_f$ is connected and $\Y \hookrightarrow J$ then $\confdim(J)=1$ and the conformal dimension is not attained. 
% \end{corollary}

% The Julia set $\Julia_f$ of a semi-hyperbolic rational map $f$ is locally connected if it is connected \cite[Theorem 10]{JuliaJohn}. Since the Julia set $\Julia_f$ is compact, it follows that $\Julia_f$ is path-connected. We prove the following theorem.

\begin{theorem}
\label{theorem:SullivanDict}
Let $f$ be a semi-hyperbolic rational map of degree $d$ with a connected Julia set $\Julia_f$. Suppose $\confdim(\Julia_f)=1$. Then the conformal dimension is attained if and only if $\Julia_f$ is homeomorphic to $S^1$ or $[0,1]$. Moreover, if $\confdim(\Julia_f)=1$ is attained, then $f$ is a sub-hyperbolic rational map whose corresponding post-critically finite rational map is $z^d$, $1/z^d$, the degree-$d$ Chebyshev polynomial or the degree-$d$ negated Chebyshev polynomial up to conjugation by M\"{o}bius transformation.
\end{theorem}

The degree-$d$ negated Chebyshev polynomials are (-1) times the degree-$d$ Chebyshev polynomials.

Since $\ARConfdim\ge \confdim$ in general, Theorem \ref{theorem:SullivanDict} holds even if we change the conformal dimension to Ahlfors-regular conformal dimension. Theorem \ref{theorem:SullivanDict} may be considered as a complex dynamical analogue of Theorem \ref{thm:BK_Rigidity} for one dimension.
\medskip
%\begin{thm}\label{theorem:SullivanDict2}
%Let $Z$ be a compact, Ahlfors $1$-regular metric space of topological dimension $1$. Suppose $G \lefttorightarrow Z$ is a uniformly quasi-M\"obius action of a group $G$ on $Z$, where the induced action $G \lefttorightarrow \operatorname{Tri}(Z)$ is cocompact. Then $G \lefttorightarrow Z$ is quasi-symmetric conjugate to an action of $G$ on the circle $S^1$ by M\"obius transformation. 
%If $G$ is a hyperbolic group and the $\ARConfdim(\partial_\infty G) 1$ is attained, then $G$ is (up to finite index) the fundamental group of closed hyperbolic 2-manifold.
%\end{thm}

Julia sets of semi-hyperbolic quadratic polynomials are related to fractals generated by classical iterated function systems \cite{Rohde:IFS}. For instance, the Julia set of $z^2+i$ is shown to be quasi-symmetric equivalent to a universal object called the {\em continuum self-similar tree}, see \cite{BM:geodesictree} and \cite{BM:uniformlybranchingtrees}. The following is then an immediate corollary of Theorem \ref{theorem:main}.

\begin{corollary}
\label{corollary:tree}
The conformal dimension of the continuum self-similar tree, which is equal to $1$, is not attained.
\end{corollary}

% \subsection*{Organization} In Section \ref{sec:Background} we review first recalls the definitions of coarse expanding conformal maps from \cite{kmp:ph:cxci}. A collection of technical results about cxc is listed in Proposition \ref{prop:visual}. Section 3 recalls the definition of antenna sets and spaces from \cite{Azzam_WigglySpace} and contains a result from \cite{Azzam_WigglySpace} which gives a criterion for when the Hausdorff dimension of a space is bigger than $1$. Section 4 concludes the proof of Theorem \ref{theorem:main} and Theorem \ref{theorem:SullivanDict}.
% \medskip 

\subsection*{Acknowledgements} The authors are grateful to Kevin Pilgrim and Dylan Thurston for useful conversation.

The first author was supported by the Simons Foundation Institute Grant Award ID 507536 while the author was in residence at the Institute for Computational and Experimental Research in Mathematics in Providence, RI.

\section{Background}\label{sec:Background}

\subsection{Finite branched coverings (fbc's)}  
\label{subsecn:fbcs} In this subsection, we assume spaces $X$ and $Z$ are compact, Hausdorff, connected, and locally connected topological spaces, and $f: X \to Z$ is a finite-to-one continuous map. We define the {\em degree} of $f$ by
\[
\deg(f):=\sup\{\# f^{-1}(z) : z \in Z\}
\]
and the {\em local degree} of $f$ at $x\in X$ by
\[
\deg(f;x):=\inf_U\sup\{\# f^{-1}(z) \cap U : z \in f(U)\}
\]
where the infimum is taken over all neighborhoods $U$ of $x$.

\begin{defn}[Finite branched covering]
The map $f: X \to Z$ is a {\em finite branched covering}
(abbreviated fbc) provided $\deg(f)<\infty$ and 
\begin{itemize} 
\item[(i)]
\[
\sum_{x \in f^{-1}(z)} \deg(f;x) = \deg(f)
\]
holds for each $z \in Z$;
\item[(ii)] for every $x_0 \in X$, there are compact neighborhoods
$U$ and $V$ of $x_0$ and $f(x_0)$ respectively such that
\[ \sum_{x \in U, f(x)=z} \deg(f; x) = \deg(f; x_0)\]
for all $z\in V$.
\end{itemize}  
\end{defn}

The following properties of fbc's are shown in \cite[Lemma 2.1.2 and Lemma 2.1.3]{kmp:ph:cxci}: For an fbc $f:X \to Z$,
\begin{itemize}
    \item $f$ is open, closed, onto, and proper;
    \item the set of {\em branch points} $B_f:= \{x \in X : \deg(f;x)>1\}$ is nowhere dense in $X$;
    \item the set of {\em branch values} $V_f:=f(B_f)$ is nowhere dense in $Z$;
    \item if $U \subset Z$ is open and connected, then its inverse image $f^{-1}(U)$ is a union of disjoint open subsets $f^{-1}(U)=\wtU_1 \sqcup \ldots \sqcup \wtU_m$ where $f: \wtU_i \to U$ is an fbc of degree $d_i$, and $d_1+\ldots+d_k=\deg(f)$.  
\end{itemize}
We refer the reader to \cite{edmonds:fbc} and \cite{kmp:ph:cxci} for more details on finite branched coverings.

\begin{lem}[Path-lifting for fbc's]
\label{lemma:path-lifting}
Finite branched covers $f:X\to Z$ have the path-lifting property: For a continuous map $\gamma: [0,1] \to Z$ and $x_0 \in f^{-1}(\gamma(0))$, there exists a continuous map $\widetilde{\gamma}: [0,1] \to X$ with $\widetilde{\gamma}(0)=x_0$ and $f \circ \widetilde{\gamma}=\gamma$.
\end{lem}

\begin{proof} An fbc is ``interior'' (sending open sets to open sets) and ``light'' (fibers being totally disconnected); cf. \cite{Wallace_QuasimonotoneTransform}. The lemma follows from \cite[Theorem 2]{Floyd_InteriorMaps}.
\end{proof}

Recall that $o$ denotes the center of the space $\Y$, i.e., the point where the three arms meet. 

\begin{cor}
\label{cor:ylifting}
If $f: X \to Z$ is an fbc and $\gamma: \Y \hookrightarrow Z$ and $x_0 \in f^{-1}(\gamma(o))$ then there is a lift $\widetilde{\gamma}: \Y \hookrightarrow X$ with $\widetilde{\gamma}(o)=x_0$ and $f\circ \widetilde{\gamma}=\gamma$. 
\end{cor}

\subsection{Coarse expanding conformal systems}

When we consider a topological cxc system $f: X \to X$, we always assume that $X$ is non-singleton, compact, connected, locally connected, and metrizable.

Let $\UUU_0$ be a cover of $X$ consisting of finitely many open connected subsets. For any $n>0$, we inductively define a cover $\UUU_n$ by the collection of connected components of inverse images of elements of $\UUU_n$ under $f$. We define $\mathbf{U}:=\bigcup_{n \geq 0} \UUU_n$ the collection of all open sets in the covers $\UUU_n$'s. If $U \in \UUU_n$, then we say that the {\em level} of $U$ is $n$ and write $|U|=n$.

\subsection*{Topological cxc systems} The dynamical system $f: X \to X$ is \emph{topological cxc} if there is a finite open cover $\UUU_0$ satisfying the following three properties:

\begin{itemize}
    \item (Expansion) $\mesh(\UUU_n)\to 0$ as $n\to \infty$. More precisely, for any open cover $\mathcal{V}$ of $X$, there is $N>0$ such that for any $U\in \UUU_n$ with $n>N$ there is $V\in \mathcal{V}$ satisfying $U\subset V$. If $X$ is a metric space, it is equivalent to $\max\{\diam U : U \in \UUU_n\}\to 0$ as $n\to \infty$.
    \item (Degree)
    \[
        \max\{\deg(f^k:\widetilde{U}\to U) : k\in \N,~\widetilde{U}\in\UUU_k,~U\in\UUU_0 \}<\infty.
    \]
    \item (Irreducibility) For any open set $W \subset X$, there exists $n\in N$ with $f^n(W)=X$.
\end{itemize}

\subsection*{Metric cxc systems}

For a metric space $(X,d_X)$, a dynamical system $f: X \to X$ is \emph{metric cxc} if it is topological cxc about a finite open cover $\UUU_0$ and, in addition, the following conditions hold for the same $\UUU_0$: There exist
\begin{itemize}
\item continuous, increasing embeddings $\rho_{\pm}:[1,\infty) \to [1,\infty)$, called the {\em forward and backward
roundness distortion functions}, and
\item increasing homeomorphisms $\delta_{\pm}:[0,1] \to [0,1]$, called the {\em forward and backward relative
diameter distortion functions},
\end{itemize}
such that the following two properties hold.
\begin{itemize}
\item (Roundness distortion) For all $n, k \in \bbN$ and for all $U \in \UUU_n,~\wtU \in \UUU_{n+k},~\tilde{y} \in \wtU$, and $y \in U$, if $f^{\circ k}(\wtU) = U$ and $f^{\circ k}(\tilde{y}) = y$, 
then we have the {\em backward roundness bound}
\[
    \roundness(\wtU, \tilde{y}) < \rho_-(\roundness(U,y)),
\]
and the {\em forward roundness bound}
\[
    \roundness(U,y) < \rho_+(\roundness(\wtU, \tilde{y})),
\]
where for an interior point $a$ of $A$ the \textit{roundness of $A$ about $a$} is defined as
\[
\roundness(A, a) = \frac{ \sup\{ |x-a|: x \in A\}}{\sup\{r: B(a,r) \subset A\}}.
\]
\item (Diameter distortion) For all $n_0, n_1, k \in \bbN$ and for all $U \in \UUU_{n_0},~U' \in \UUU_{n_1},~\wtU \in \UUU_{n_0+k}$, and $\wtU'\in \UUU_{n_1+k}$ with $\wtU' \subset \wtU$ and $U' \subset U$,
if $f^k(\wtU) = U$ and $f^k(\wtU') = U'$, then
\[
\frac{\diam\wtU'}{\diam\wtU} < \delta_-\left(\frac{\diam U'}{\diam U}\right)
\]
and
\[
\frac{\diam U'}{\diam U} < \delta_+\left(\frac{\diam\wtU'}{\diam\wtU}\right).
\]
\angela{
In other words:  given two nested elements of $\mathbf{U}$, iterates of $f$ both forward and backward distort their relative sizes by an amount independent of the iterate.
As a consequence, one has then also the {\em backward upper and lower relative diameter
bounds}:
\begin{equation}
\label{eqn:brdb}
\delta_+^{-1}\left(\frac{\diam U'}{\diam U}\right) <
\frac{\diam\wtU'}{\diam \wtU} < \delta_-\left(\frac{\diam U'}{\diam
U}\right)
\end{equation}
and the {\em forward upper and lower relative diameter bounds}:
\begin{equation}
\label{eqn:frdb}
\delta_-^{-1}\left(\frac{\diam\wtU'}{\diam\wtU}\right) < \frac{\diam
U'}{\diam U} < \delta_+\left(\frac{\diam\wtU'}{\diam\wtU}\right).
\end{equation}
}
\end{itemize}
From \cite{kmp:ph:cxci}, we have that the property of being metric cxc is preserved by quasi-symmetric conjugacies. %Also, if $f: X \to X$ is metric cxc, then $X$ is doubling, hence by Assouad's embedding theorem admits a quasi-symmetric embedding into some finite-dimensional Euclidean space \textcolor{red}{need reference}, hence into $\ell^\infty$. 

\subsection*{Visual metrics}
\label{subsecn:visual}

%Axiom [Expansion] implies that the maximum diameters of the elements of $\UUU_n$ tend to zero uniformly in $n$.   Since $\UUU_0$ is assumed finite, each covering $\UUU_n$ is finite, so for each $n$ there is a  minimum diameter of an element of $\UUU_n$.  Since $X$ is perfect and, by assumption, each $U \in \mathbf{U}$ contains a point of $X$, each $U$ contains many points of $X$ and so has positive diameter.   Hence there exist decreasing positive sequences $c_n, d_n \to 0$ such that the {\em diameter bounds} hold:
%\begin{equation}
%\label{eqn:diameter_bounds}
%0 < c_n \leq  \inf_{U \in \UUU_n} \diam U \leq \sup_{U \in \UUU_n} \diam U \leq  d_n.
%\end{equation}
%
A dynamical system $f: X \to X$ that is topological cxc with respect to a finite open cover $\UUU_0$ determines a class of so-called \emph{visual metrics}. Visual metrics are compatible with the dynamics and $\UUU_0$ in a sense that any element in $\bf{U}$ can be thought of as a ball whose size exponentially decays as the level increases. The following results summarize key properties of the visual metrics; see \cite[Chapter 3]{kmp:ph:cxci}.

\begin{prop}[Visual metric, {\cite[Proposition 3.3.2 and Proposition 3.2.3]{kmp:ph:cxci}}]
\label{prop:visual}
Let $(X,d_X)$ be a metric space. Suppose that $f: X \to X$ is metric cxc with respect to an open cover $\UUU_0$. There exist $\epsilon>0$ and a metric $d$, called a visual metric, on $X$ that is quasi-symmetric to $d_X$ such that the following estimates hold for $(X,d)$:
\begin{enumerate}
\item (Nearly balls, I) There is some constant $C>1$ such that, for all $W\in \mathbf{U}$, there is a point $x \in W$ so that
$$B(x,(1/C)e^{-\epsilon|W|}) \subset   
W \subset B(x,C e^{-\epsilon|W|}).$$
\item (Nearly balls, II) There is a radius $r_1>0$ such that, for any $n\geq 1$ and for any $x \in X$, there
is some $W\in \UUU_n$ so that $B(x,r_1 e^{-\epsilon n})\subset W$.
\item  (Nearly balls, III) There exists $r_0$ such that, for any $r\in (0,r_0)$ and any  $x \in X$,
there exist $W$ and $W'$ in $\mathbf{U}$ such that $|W|-|W'|=O(1)$,
\[
x \in W'\subset B(x,r)\subset W,
\]
and $$\max\{\roundness(W,x),\roundness(W',x)\}=O(1).$$
\item (Homothety) For every $x,y \in X$, $d(x,y) \geq e^{-\vep} \cdot d(f(x), f(y)) $. 
%\item (Balls map to balls) For $0<r<r_0$ as above, $f(B(x,r))=B(f(x), e^\epsilon r)$. 
%\item (Homothety, given Koebe space) Furthermore, for each $n \in \mathbb{N}$, if $f^n: B(x,r) \to B(f^n(x), e^{n\epsilon} r)$ is a homeomorphism, then for all $\zeta_1, \zeta_2 \in B(x, r/4)$, 
%\[ d(f^n(\zeta_1), f^n(\zeta_2)) = e^{n\varepsilon}d(\zeta_1,\zeta_2).\]
\end{enumerate}
\end{prop}
In fact, we have $d(x,y)=e^{-\vep} \cdot d(f(x),f(y))$ for $x,y\in B$ where $B$ is a sufficiently small ball that does not contain a branch point \cite[Proposition 3.2.3]{kmp:ph:cxci}. 

The following proposition implies that a visual metric is a canonical metric, in some sense, with which the topological cxc map becomes metric cxc.

\begin{prop}[{\cite[Theorem 3.5.3 and Corollary 3.5.4]{kmp:ph:cxci}}]\label{prop:CXCConj}
Let $f:X\to X$ be a topological cxc dynamical system. There is a quasi-symmetric equivalence class $\mathcal{G}(f)$ of metrics on $X$, called the canonical gauge, such that
\begin{itemize}
    \item visual metrics are in $\mathcal{G}(f)$, and
    \item $d\in \mathcal{G}(f)$ if and only if $f$ is metric cxc with respect to $d$.
\end{itemize}
Hence, if two metric cxc dynamical systems $f:X\to X$ and $g:Z \to Z$ are topologically conjugate by a homeomorphism $h:X\to Z$, then $h$ is a quasi-symmetric conjugacy.
\end{prop}

We call the quasi-symmetric equivalence class $\mathcal{G}(f)$ of metrics on $X$ that contains visual metrics the {\em canonical gauge} (or the {\em conformal gauge}) of $f$. The conformal dimension $\confdim((X,\mathcal{G}(f)))$ (resp.\@ the Ahlfors-regular conformal dimension $\ARConfdim((X,\mathcal{G}(f)))$) denotes the infimal Hausdorff dimension of metrics (resp.\@ Ahlfors-regular metrics) on $X$ in the canonical gauge $\mathcal{G}(f)$.

\subsection{Antenna-like spaces}\label{sec:Antenna}

We refer the reader to \cite{Azzam_WigglySpace} for a detailed account of this subsection.

\begin{defn}[Antenna-like space]
\label{defn:antenna} Suppose $X$ is a compact connected metric space. For $0<c<1$ and an open subset $U$ of $X$, we say that \emph{$U$ has a $c$-antenna} if there is a homeomorphism $h: \Y \hookrightarrow U$ such that for all permutations $(i,j,k)$ of $(1,2,3)$, the distance between $h(e_i)$ and $h([0,e_j])\cup h([0,e_k])$ is at least $c\cdot \diam(U)$. The space $X$ is called \emph{$c$-antenna-like} if for each $r<\frac{1}{2}\diam(X)$, every ball $B(x,r)$ has a $c$-antenna.
\end{defn}

%Note that the property of a ball having a $c$-antenna is hypothesis only on the image $h(Y)$, and not on the parametrization $h$.  If a ball has a $c$-antenna, then the corresponding antenna $h(Y)$ has diameter in $[cr, 2r]$. The property of a ball having a $c$-antenna is invariant under scaling the metric. 
The following theorem is a combination of the results in \cite{Azzam_WigglySpace}. 

\begin{thm}[Dimension of antenna-like spaces \cite{Azzam_WigglySpace}]
\label{thm:dimant}
Suppose $X$ is a compact connected metric space. Suppose that $X$ is  $c$-antenna-like for some $c\in(0,1)$. Then the following are hold.
\begin{itemize}
    \item If a metric space $Z$ is quasi-symmetric to $X$, then $Z$ is also a $c'$-antenna-like for some $c'\in(0,1)$.
    \item There is a uniform constant $b>0$ that is independent of $c$ such that $\hdim(X)> 1+bc^2$.
    \item Hence, $\confdim(X)>1$ or $\confdim(X)=1$ but not attained.
\end{itemize} 
\end{thm}

\section{No exotic 1-dimensional cxc systems}\label{sec:NoExotic1DimCxc}

Cxc maps on a circle $S^1$ or a closed interval $I$ are quasi-symmetrically conjugate to the well-known dynamical systems.

\begin{thm}\label{thm:CxcOneDim}
Let $f:X\to X$ be a metric cxc map of degree-$d$.
\begin{itemize}
    \item[(i)] If $X=S^1$, then $f$ is quasi-symmetrically conjugate to the map $z\mapsto z^d$ on the unit circle $\mathbb{S}^1$ \cite[Theorem 4.1.1]{kmp:ph:cxci}.
    \item[(ii)] If $X=I$, then $f$ is quasi-symmetrically conjugate to the degree-$d$ Chebyshev or negated Chebyshev polynomial on $[-1,1]$.
\end{itemize}
\end{thm}

As an immediate corollary, we can obtain a rigidity theorem, promoting semi-hyperbolic rational maps to sub-hyperbolic rational maps, i.e., every critical point in Julia sets is preperiodic.

\begin{cor}\label{cor:RigiditySemiSub}
Let $f$ be a semi-hyperbolic rational map with the Julia set $\Julia_f$ homeomorphic to $S^1$ or $I$. Then $f$ is a sub-hyperbolic rational map.
\end{cor}

\subsection*{Chebyshev and negated Chebyshev polynomials}
For $d\ge 1$, the degree-$d$ Chebyshev polynomial $T_d$ is defined by $T_d(\cos{\theta})=\cos(d\cdot\theta)$. The Julia set of $T_d$ is a closed interval $[-1,1]$ so that $T_d:[-1,1]\to[-1,1]$ is well-defined. We always have $T_d(1)=1$ but $T_d(-1)=-1$ if $d$ is odd and $T_d(-1)=1$ if $d$ is even.
\medskip

\noindent{\em Case (i): $d=2n+1$.}

We have
\[
f^{-1}(-1)=\{-1,x_1,x_2,\dots,x_n\}
\]
and
\[
f^{-1}(1)=\{1,y_1,y_2,\dots,y_n\}
\]
so that $\deg(f;x_i)=\deg(f;y_i)=2$ for any $i,j$. By reordering the indices, we have
\[
    -1<y_1<x_1<y_2<x_2<\cdots<y_n<x_n<1.
\]
Then $f$ homeomorphically maps each connected component of $[-1,1]\setminus f^{-1}(\{-1,1\})$ to $(-1,1)$.
\medskip

\noindent{\em Case (ii): $d=2n$.}

We have
\[
f^{-1}(-1)=\{x_1,x_2,\dots,x_n\}
\]
and
\[
f^{-1}(1)=\{-1,1,y_1,y_2,\dots,y_{n-1}\}
\]
so that $\deg(f;x_i)=\deg(f;y_j)=2$ for any $i,j$. By reordering the indices, we have
\[
    -1<x_1<y_1<x_2<y_2<\cdots<y_{n-1}<x_n<1.
\]
Again $f$ homeomorphically maps each connected component of $[-1,1]\setminus f^{-1}(\{-1,1\})$ to $(-1,1)$.
\medskip

Now we consider the {\em negated Chebyshev polynomial} $T'_d:=-T_d$. If $d=2n$ then $T'_d$ and $T_d$ are conjugate each other by the negation $z\mapsto -z$. Suppose $d=2n+1$. Since $T'_d$ swaps the two end points $\{-1,1\}$, $T'_d$ is not conjugate to $T_d$. With the notations used in Case (i) above, we have $f^{-1}(-1)=\{1,y_1,y_2,\dots,y_n\}$ and $f^{-1}(1)=\{-1,x_1,x_2,\dots,x_n\}$.
\medskip

One can construct the degree-$d$ Chebyshev polynomial $T_d:[-1,1] \to [-1,1]$ by projecting the map $z\mapsto z^d$ on the unit circle $\mathbb{S}^1$ onto the diameter $[-1,1]$ on the real axis. Suppose $d=2n+1$. The map $\exp(i\theta)\mapsto \exp(i(\pi-\theta))$ is the reflection of $\mathbb{S}^1$ in the imaginary axis. It follows from
\[
    \exp(i \cdot (2n+1)\cdot (\pi-\theta))=\exp(i\cdot(\pi-(2n+1)\theta))
\]
that the projection of the map $z\mapsto z^d$ on $\mathbb{S}^1$ onto the diameter on the imaginary axis $[-i,i]$ is well-defined. This vertical projection gives rises to the negated Chebyshev polynomial dynamics.

\begin{proof}[Proof of Theorem \ref{thm:CxcOneDim}]
See \cite[Theorem 4.1.1] {kmp:ph:cxci} for the proof of Theorem \ref{thm:CxcOneDim}-(i). We prove Theorem \ref{thm:CxcOneDim}-(ii) here. Both proofs show the existence of a semi-conjugacy first and then show the semi-conjugacy is indeed a conjugacy.

By Proposition \ref{prop:CXCConj}, it suffices to show that $f$ is topologically conjugate to the dynamics of a Chebyshev or negated Chebyshev polynomial on the interval.

Let $I=[0,1]$. It follows from the topology of $I$ that for every branch point $x$ of $f$, we have $\deg(f;x)=2$, $x \in \interior(I)$, and $f(x)\in\{0,1\}$. Also, since $f$ is an open map, $f(\{0,1\})\subset \{0,1\}$. Hence every branch point is prefixed.

There are three cases: (i) $f(0)=0$ and $f(1)=1$, (ii) $f(0)=1$ and $f(1)=0$, and (iii) $f(1)=f(0)=0~\mathrm{or}~1$.

Let $x_1,x_2,\dots,x_n\in(0,1) \cap f^{-1}(0)$ and $y_1,y_2,\dots,y_m\in(0,1) \cap f^{-1}(1)$ so that $x_i<x_{i+1}$ and $y_j<y_{j+1}$ for any $i$ and $j$. Then $\deg(f;x_i)=\deg(f;y_j)=2$ for any $i,j$.
\medskip

\noindent {\it Case (i)}: Since
\[
    \deg(f;0)+\sum_{i=1}^n \deg(f;x_i)=\deg(f;1)+\sum_{i=1}^m \deg(f;y_i),
\]
we have $n=m$ and $\deg(f)=2n+1$. It follows that
\[
    0<y_1<x_1<y_2<x_2<\cdots<y_n<x_n<1.
\]
The map $f$ is then a homeomorphism onto $(0,1)$ on each connected component $J$ of $(0,1)\setminus \bigcup_{i=1}^n\{x_i,y_i\}$; if not, then $J$ has a branch point $z$ which must be mapped to either $0$ or $1$ so that $z$ is $x_i$ or $y_i$ for some $i$.

One can compute the asymptotic growth rate $s$ of the number of laps for $f^n$ (simply called the {\em growth number}), which is equal to the leading eigenvalue of the incidence matrix for the Markov partition of $f:I\to I$ by $\bigcup_{i=1}^n\{x_i,y_i\}$, and check that $s>1$. Then there is one and only one piecewise linear map $g:I \to I$ with $|g'(x)|=s$ for every $x$ where $g'(x)$ is well-defined such that there is a semi-conjugacy $h$ from $f:I\to I$ to $g:I\to I$ \cite[Theorem 7.4]{MinorThurston_Kneading}. Note that (Irreducibility) for topological cxc maps implies topological transitivity, i.e., for any pair of open sets $U, V\subset I$, we have $f^n(U) \cap V\neq \emptyset$ for some $n\ge0$. Then the semi-conjugacy $h$ is a topological conjugacy \cite[Proposition 4.6.9]{ALM_CombDynEntOndDim}.

Since $f$ and the Chebyshev polynomial $T_{2n+1}$ have the same transition matrices of Markov partitions, the growth numbers of $f$ and $T_{2n+1}$ are equal. It follows from the same argument used for $f$ that $T_{2n+1}$ is also topologically conjugate to $g$. Hence $f$ and $T_{2n+1}$ are topologically conjugate.
\medskip

\noindent {\it Case (ii)}: By a similar argument in Case (i), we can show that $f$ is topologically conjugate to the negated Chebyshev polynomial $T'_{2n+1}$ of degree $2n+1$.
\medskip

\noindent {\it Case (iii)}: Without loss of generality, suppose $f(0)=f(1)=1$. Since
\[
    \sum_{i=1}^n \deg(f;x_i)=\deg(f;0)+\deg(f;1)+\sum_{i=1}^m \deg(f;y_i),
\]
we have $n=m+1$ and $\deg(f)=2n$. It follows that
\[
0<x_1<y_1<x_2<\cdots<y_{n-1}<x_n<1.
\]
Then by a similar argument used in Case (i), one can show that $f$ is topologically conjugate to a Chebyshev polynomial $T_{2n}$ of degree $2n$.

\end{proof}

\section{Proof of Theorem \ref{theorem:main} and Theorem \ref{theorem:SullivanDict}}

In this section we prove Theorems \ref{theorem:main} and \ref{theorem:SullivanDict}.

\subsection{Proof of Theorem \ref{theorem:main}}
Suppose $f:X\to X$ is metric cxc with respect to $\UUU_0$ and $X$ is equipped with properties (1)-(4) in Proposition \ref{prop:visual}. We also use the constants $C, r_0$, and $\vep$ for used in Proposition \ref{prop:visual}. Suppose that there is an embedding $\Y \hookrightarrow X$.

By Theorem \ref{thm:dimant}, it suffices to show that $X$ is $c''$-antenna-like for some $c''>0$. We divide the proof into three steps.
\begin{itemize}
    \item [(1)] We first show in Lemma \ref{lemma:antenna_1} that there is $c>0$ so that every $U\in \UUU_0$ has a $c$-antenna.
    \item [(2)] Then we can use Lemma \ref{lemma:antenna_2} to show that there exists $c'>0$ so that every $U\in \mathbf{U}$ has a $c'$-antenna.
    \item [(3)] Then Lemma \ref{lemma:antenna_3} applies to show that $X$ is $c''$-antenna-like for some $c''>0$.
\end{itemize} 

Let us show the three lemmas.

\begin{lem}\label{lemma:antenna_1}
There exists $c>0$ so that every $U \in \UUU_0$ has a $c$-antenna.
\end{lem}
\begin{proof}
By the assumption, we have an embedding $\Y\hookrightarrow X$. There exists $U_0 \in \UUU_0$ containing the center $o$ of $\Y$. By restricting the copy of $\Y$ to $U_0$ if necessary, we may assume that $\Y \hookrightarrow U_0$. Since $\UUU_0$ is a finite cover, it suffices to show that every $V\in \UUU_0$ has an antenna.

Let $V\in \UUU_0$. We show that there is a connected component $U_{N(V)}$ of $f^{-N(V)}(U_0)$ so that $U_{N(V)}\subset V$. It follows from Proposition \ref{prop:visual} that there is $x\in V$ so that $B(x,1/C)\subset V$. We take an $r$-ball $B(x,r)\subset V$ for $r<1/C$. By (Irreducibility) of topological cxc maps, there is $K>0$ so that $f^K(B(x,r))=X$. Then, there is a connected component $U_K$ of $f^{-K}(U_0)$ so that $U_K \cap B(x,r) \neq \emptyset$. By Proposition \ref{prop:visual}-(1) we have 
\[
\diam (U_K) < 2Ce^{-\varepsilon K}.
\]
Hence, by triangle inequality, we have
\[
U_K \subset B(x, r+2C e^{-\varepsilon K}).
\]
Note that we can make $r$ arbitrarily small. Moreover, if we decrease $r$, then $K$ should increase. Hence we can make $K$ arbitrarily large. For a sufficiently small $r$ and a sufficiently large $K$, we have $r+2C e^{-\epsilon\cdot K}< 1/C$ so that $U_K \subset V$. We write $K=K(V)$ to indicate the dependence on $V$.

Since $f^{K(V)}:U_{K(V)}\to U_0$ is a fbc, using Corollary \ref{cor:ylifting}, we can find a lift $\Y \hookrightarrow U_{K(V)}\subset V$ of $\Y\hookrightarrow U_0$. Thus any $V\in \UUU_0$ has an antenna.
\end{proof}

\begin{lem}\label{lemma:antenna_2}
There exists $c' \in (0,1)$ such that every $U' \in \mathbf{U}$ has a $c'$-antenna.
\end{lem}
\begin{proof}
It suffices to show that for any $U\in \UUU_0$ and $n>0$, every $U'\in \UUU_n$ with $f^n(U')=U$ has a $c'$-antenna.

Suppose $h:\Y \hookrightarrow U$ is an embedding such that for all permutations $(i,j,k)$ of $(1,2,3)$, the distance between $h(e_i)$ and $h([0,e_j])\cup h([0,e_k])$ is at least $c\cdot \diam(U)$. 
Suppose $U'\in \UUU_n$ is a connected component of $f^{-n}(U)$. Let $y = h(o)$ where $o$ is the branch point $\Y$. Let $x \in U'$ be a point such that $f^n(x) = y$. Since $f^n:U' \to U$ is an fbc, it follows from Corollary \ref{cor:ylifting} that there is an embedding $\Tilde{h}:\Y \to U'$ which is a lift of $h:\Y \hookrightarrow U$, i.e., $f^n \circ \Tilde{h} = h$ and $\Tilde{h}(o) = x$. Proposition \ref{prop:visual}-(4) implies that for any $t_1, t_2 \in \Y$, 
\[ d(\Tilde{h}(t_1), \Tilde{h}(t_2)) \geq e^{-\vep n} d(h(t_1), h(t_2)). \]
Thus
\begin{align*} 
&\hspace{12pt} d(\Tilde{h}(e_1), \Tilde{h}([0,e_2]\cup[0,e_3]))\\
&= \inf\{ d(\Tilde{h}(e_1), \Tilde{h}(t)): t \in [0,e_2]\cup[0,e_3]\} \\
&\geq e^{-\vep n}  \inf\{ d(h(e_1), h(t)):t \in [0,e_2]\cup[0,e_3]\} \\
&= e^{-\vep n} d(h(e_1), h([0,e_2]\cup[0,e_3])) \\
&\geq e^{-\vep n} c\cdot \diam U. 
\end{align*}
Likewise, we have
\begin{align*} 
d(\Tilde{h}(e_2), \Tilde{h}([0,e_3]\cup[0,e_1])), d(\Tilde{h}(e_3), \Tilde{h}([0,e_1]\cup[0,e_2]))
\geq e^{-\vep n} c \cdot \diam U. 
\end{align*}
It follows from Proposition \ref{prop:visual}-(1) that 
\[
    \diam U' \leq 2Ce^{-\vep n} \leq C^2e^{-\vep n}\cdot\diam U.
\]
We conclude that $U'$ has a $c'$-antenna for $c' = c/{C^2}$. 
\end{proof}

\begin{lem}\label{lemma:antenna_3}
If every $U\in \mathbf{U}$ has a $c$-antenna for some $c>0$, then $X$ is $c'$-antenna-like for some $c' \in (0,1)$. 
\end{lem}
\begin{proof}
Suppose $B(x,r)$ be a ball of radius $r$ Let $r_0$ is as defined in Proposition \ref{prop:visual}.

If $r<r_0$, then by Proposition \ref{prop:visual}-(3) there exist $W, W'\in \mathbf{U}$ such that $|W|-|W'| = O(1)$ and 
\[ W' \subset B(x,r) \subset W. \]
Since $|W| - |W'| = O(1)$, Proposition \ref{prop:visual}-(1) yields
\[ \diam W' \geq C^{-2} e^{-\vep (|W|'-|W|)} \diam W \geq 2rC^{-2} e^{-\vep\cdot O(1)}. \]
Let $c' = c \cdot C^{-2}\cdot  e^{-\vep\cdot O(1)}$. By the assumption in the statement of the lemma, there exists an embedding $h:\Y \to W'\subset B(x,r)$ such that for all permutations $(i,j,k)$ of $(1,2,3)$, the distance between $h(e_i)$ and $h([0,e_j])\cup h([0,e_k])$ is at least $c\cdot \diam(W') \geq c'\cdot 2r=c'\cdot \diam B(x,r)$, i.e., $B(x,r)$ has a $c'$-antenna. If $r_0> \frac{1}{2}\diam X$, then this completes the proof of the lemma.

Suppose now $r_0<r < \frac{1}{2}\diam X$. Then $\frac{2rr_0}{\diam X} < r_0$ so that there exists $r' \in \R$ such that $\frac{2rr_0}{\diam X} < r' < r_0$. By the previous paragraph, there exists an embedding $h:\Y \to B(x,r')$ such that for all permutations $(i,j,k)$ of $(1,2,3)$, the distance between $h(e_i)$ and $h([0,e_j])\cup h([0,e_k])$ is at least $c'\cdot 2r'> c' \cdot (2r_0/\diam X) \cdot 2r$. Replacing $c'$ by $c' \cdot (2r_0/\diam X)$, we can make $X$ $c'$-antenna-like.
\end{proof}

\subsection{Proof of Theorem \ref{theorem:SullivanDict}}

We prove the following proposition first and then complete the proof of Theorem \ref{theorem:SullivanDict}.

\begin{prop}\label{prop:NonSmoothExistY}
Suppose $\Julia_f$ is the Julia set of a semi-hyperbolic rational map $f$. If $\Julia_f$ is connected, and $\Julia_f$ is homeomorphic to neither the circle $S^1$ or a closed interval $[0,1]$, then $\Y \hookrightarrow J$. 
\end{prop}
\begin{proof}
Let $f$ be a semi-hyperbolic rational map with a connected Julia set $\Julia_f$. Then $\Julia_f$ is locally connected \cite[Theorem 2]{JuliaJohn}. A connected and locally connected compact subset of $\R^2$ is arcwise connected \cite[Lemmas 17.17 and 17.18]{Milnor}, i.e., for every $x,y \in \Julia_f$, there exists a homeomorphic embedding, called an arc, $\gamma:[0,1] \to \Julia_f$ such that $\gamma(0) = x$ and $\gamma(1) = y$.

Let us assume that $\Julia_f$ does not contain a homeomorphic copy of $\Y$ and show that $\Julia_f$ is either $S^1$ or $[0,1]$. 

If $\Julia_f$ contains a simple closed cure $C$ that omits a point $x \in \Julia_f$, then $x$ is joined to some point in $C$ by an arc $\gamma:[0,1]\to \Julia_f$ so that $\gamma(0)=x$ and $\gamma(1)\in C$. Let $t:=\min\{t\in[0,1] : \gamma(t)\in C$. Let $I$ be a subarc of $C$ containing $\gamma(t)$ in its interior. Then $I\cup \gamma([0,t])$ is the image of an embedding $\Y\hookrightarrow \Julia_f$, which contradicts to the assumption that $\Julia_f$ does not contains $\Y$. Hence $\Julia_f$ is homeomorphic to $S^1$.

If $\Julia_f$ does not contain a simple closed loop, then $\Julia_f$ is a locally connected continuum that contains no simple closed curves, i.e., a dendrite in the sense of \cite{dendriteEquivalence}. Pick $x \in \Julia_f$. Then $x$ is either a cut point of $\Julia_f$ or an end point of $\Julia_f$ \cite[Theorem 1.1 (3)]{dendriteEquivalence}.

If $\Julia_f \backslash \{x\}$ has more than 2 connected components, then one can find a homeomorphic copy of $\Y$ centered at $x$, contradicting our assumption. If $\Julia_f \backslash\{x\}$ has two connected components, we denote by $L_1$ and $L_2$ the closures of the connected components. If $x$ is an end point of $\Julia_f$, we let $L_1 = \Julia_f$ and $L_2 = \emptyset$. 

Every point $y \in L_1$ with $y\neq x$ can be joined to $x$ by an unique arc \cite[Theorem 1.2 (20)]{dendriteEquivalence}. Let $y_1$ and $y_2$ be two distinct points in $L_1 \backslash \{x\}$. For $i = 1,2$, let $\gamma_i$ be the unique arc joining $x$ and $y_i$. Since $ x \in \gamma_1\cap \gamma_2$, $\gamma_1\cap \gamma_2$ is non-empty. By hereditary unicoherency of $\Julia_f$(\cite[Theorem 1.1 (18)]{dendriteEquivalence}), $\gamma_1 \cap \gamma_2$ is connected. Then either $\gamma_1 \cap \gamma_2 = \gamma_1$ or $\gamma_1 \cap \gamma_2 = \gamma_2$, for otherwise $\gamma_1 \cup \gamma_2$ would contain a homeomorphic copy of $\Y$. This proves that $L_1$ is homeomorphic to $[0,1]$ if they are non-empty, with $x$ as one end point. The same proof shows that $L_2$ is either empty or homeomorphic to $[0,1]$ with $x$ as one end point, and hence $\Julia_f$ is homeomorphic to $[0,1]$. 
\end{proof}

\begin{proof}[Proof of Theorem \ref{theorem:SullivanDict}]

Suppose $f$ is a degree-$d$ semi-hyperbolic rational map with connected Julia set $\Julia_f$ such that $\confdim(\Julia_f)=1$. By Theorem \ref{theorem:main} and Proposition \ref{prop:NonSmoothExistY}, the conformal dimension is attained if and only if $\Julia_f$ is homeomorphic to $S^1$ or $[0,1]$. If $\Julia_f$ is homeomorphic to $S^1$ or $[0,1]$, then, by Theorem \ref{thm:CxcOneDim} and Corollary \ref{cor:RigiditySemiSub}, $f$ is sub-hyperbolic whose dynamics restricted to $\Julia_f$ is quasi-symmetrically conjugate to $z^d$, the degree-$d$ Chebyshev polynomial, or the degree-$d$ negated Chebyshev polynomial. We can also have $1/z^d$ if the quasi-symmetric conjugacy reverses the orientation of on the circle. Then the unique post-critically rational map corresponding to the sub-hyperbolic rational map $f$, by \cite{McM_PCFCenter}, is one of these rational maps up to conjugation by M\"{o}bius transformations. 
\end{proof}

\bibliographystyle{alpha}
\bibliography{refs}
\end{document}